\documentclass[reqno]{amsart}
\usepackage{amsmath, amssymb, amsthm, geometry, enumerate, graphicx, amsfonts, hyperref, mathrsfs}
\hypersetup{colorlinks=true,linkcolor=red, anchorcolor=blue, citecolor=blue, urlcolor=red, filecolor=magenta, pdftoolbar=true}
\theoremstyle{plain}
\newtheorem{thm}{\it Theorem}[section]

\theoremstyle{remark}
\newtheorem{defn}[thm]{Def{}inition}
\newtheorem{rem}[thm]{Remark}
\newtheorem{exa}[thm]{Example}
\numberwithin{equation}{section}
\usepackage{enumerate}
\usepackage{xcolor}
\allowdisplaybreaks
\begin{document}
\title [Perturbation of  Hilbert-Schmidt Frames]{On perturbation of  Hilbert-Schmidt frames}

\author[Jyoti]{Jyoti}
	\address{{\bf{Jyoti}}, Department of Mathematics,
		University of Delhi, Delhi-110007, India.}
	\email{jyoti.sheoran3@gmail.com}

\author[Lalit   Kumar Vashisht]{Lalit  Kumar  Vashisht$^{*}$}
	\address{{\bf{Lalit  Kumar  Vashisht}}, Department of Mathematics,
		University of Delhi, Delhi-110007, India.}
	\email{lalitkvashisht@gmail.com}

\begin{abstract}
In this paper, we study perturbation of Hilbert-Schmidt frames under structured modifications, where the perturbation takes the form of replacing finitely or infinitely many frame elements. We establish explicit criteria under which the perturbed sequence retains the Hilbert-Schmidt frame property. In the finite case, the stability bounds depend quantitatively on the perturbation size and the number of altered elements. For the infinite case, we identify sufficient conditions ensuring stability under globally controlled perturbations. Our study  includes illustrative examples demonstrating the applicability of the results.
\end{abstract}


\keywords{Frames; Hilbert-Schmidt Weaving; Perturbation.\\
The research of Jyoti is supported by the  WISE-PDF research grant of WISE--KIRAN Division, Department of Science and Technology (DST), Government of India (Grant No.: DST/WISE-PDF/PM-6/2023(G)).  Lalit \break  Kumar Vashisht is  supported by the Faculty Research Programme Grant-IoE, University of Delhi \ (Grant No.: Ref. No./IoE/2025-26/12/FRP).\\
$^*$Corresponding author}

\maketitle

\baselineskip15pt

\section{Introduction}
The concept of frames, which was first based on work by Gabor \cite{G46} in decomposing complex signals (functions) into elementary functions, was introduced by Duffin and Schaeffer \cite{DS} and later reviewed by Young \cite{Young} in the study of several challenging problems pertaining to non-harmonic Fourier series. A countable collection  $\{x_k\}_{k \in \mathbb{I}}$ of members of a  separable Hilbert space $\mathcal{H}$ is a frame for $\mathcal{H}$ if, there exist positive real numbers $a_o$ and $b_o$, called the \emph{lower frame bound} and \emph{upper frame bound} of $\{x_k\}_{k \in \mathbb{I}}$,  respectively,  such that
\begin{align}\label{b01.1}
a_o \|x\|^2\leq  \sum_{k \in \mathbb{I}}|\langle x, x_k\rangle|^2 \leq b_o \|x\|^2 \ \text{for all} \ x \in \mathcal{H}.
\end{align}
In \eqref{b01.1}, if only upper inequality holds, then we say that $\{x_k\}_{k \in \mathbb{I} }$   is  a \emph{Bessel sequence} with the \emph{Bessel bound} $b_o$. Further, the bounded linear map $S: \mathcal{H} \rightarrow \mathcal{H}$   given by $Sx = \sum_{k \in \mathbb{I}}\langle x, x_k\rangle x_k$ is called the \emph{frame operator}. If $\{x_k\}_{k \in \mathbb{I} }$ is a frame for $\mathcal{H}$, then the frame operator is positive and invertible on $\mathcal{H}$. This gives a series expansion for each element of the space $\mathcal{H}$. To be precise, for each $x \in \mathcal{H}$, we have $x = SS^{-1}x =\sum_{k \in \mathbb{I}} \langle x, S^{-1}x_k \rangle x_k$. We refer to texts  by Heil \cite{H11}, Han \cite{BHan}, Young \cite{Young}, Zhang and Jorgensen  \cite{Zpelle}  for frames and their applications in engineering science and pure mathematics.

\subsection{Related work}
To broaden the concept of frames, Sun \cite{WsunI}  took into account a family of operators acting on $\mathcal{H}$  with range in closed subspaces of another Hilbert space. More precisely,  for a given collection $\{G_j\}_{j \in \mathbb{I}}$ of closed subspaces of  a separable Hilbert space $\mathcal{K}$, a countable collection $\{T_j\}_{j \in \mathbb{I}}$, where each $T_j:\mathcal{H} \rightarrow  G_j$ is a bounded linear operator,  is a \emph{$g$-frame}  for $\mathcal{H}$ with respect to $\{G_j\}_{j \in \mathbb{I}}$ if
\begin{align*}
\alpha \|x\|^2 \leq \sum\limits_{j \in \mathbb{I}}\|T_j x\|^2 \leq \beta \|x\|^2, \  x \in \mathcal{H},
\end{align*}
holds for some positive real numbers $\alpha$ and $\beta$. Sun also proved some perturbation results for $g$-frames in \cite{WsunII}. Koo and Lim \cite{Yoo} examined Schatten $p$-class operators, the space of linear compact operators whose sequences of singular values are in $\ell^p$, in terms of frame requirements as part of the development of frames connected to operator theory. Von Neumann-Schatten $p$-frames in separable Banach spaces are covered in \cite{SA}.  In \cite{JVPDF1}, the authors recently examined Riesz bases and Hilbert-Schmidt frames for separable Hilbert spaces.
For relations between Hilbert-Schmidt operators and frames in separable Hilbert spaces, we refer to \cite{Balaza}. The books by Schatten \cite{SCHI} and Simon \cite{B.simon} are excellent resources for learning the basic characteristics of Hilbert-Schmidt operators. The paper \cite{WZhang} has several duality relations for Hilbert-Schmidt frames in separable Hilbert spaces. In \cite{LG}, G\v{a}vruta  presented frames for subspaces in terms of bounded linear operators on separable Hilbert spaces. Frames for subspaces were further studied in a series of published papers \cite{DVV17, DV17ar, XZG}.
Hilbert-Schmidt frames (see Sec.\,\ref{sec2}), consisting of sequences of bounded linear operators from a Hilbert space $\mathcal{H}$ into the Hilbert-Schmidt class $\mathcal{C}_2 \subseteq \mathcal{B}(\mathcal{K})$, have emerged as a powerful generalization of vector frames, particularly useful in analyzing operator-valued data and matrix-valued signal systems. In \cite{JV23}, the authors studied G\v{a}vruta's type frame conditions  for a class of Hilbert-Schmidt operators, in which a bounded linear operator on the underlying Hilbert space controls the  lower frame condition. They also gave some frame-preserving mappings for Hilbert-Schmidt frames for subspaces of a separable Hilbert space. The existence of  Hilbert-Schmidt frames  for  subspaces  of the Hilbert-Schmidt class $\mathcal{C}_2$ is also established in \cite{JV23}.
They proved in  \cite{JV23} that every separable Hilbert space admits a Hilbert-Schmidt frame with respect to a given separable Hilbert space. Fundamental properties of Hilbert-Schmidt Riesz bases in separable Hilbert spaces can be found in  \cite{JV23}.

In recent years, mathematicians and physicists have been interested in studying the stability of frames in Hilbert spaces, see \cite{GossonI, GossonII, H11, Young}. To be exact, fundamental properties associated with frame conditions must be preserved under perturbations in many applications of frames, such as  signal processing \cite{BHan, H11}, quantum physics \cite{GossonI}, time-frequency analysis \cite{KG}, distributed signal processing \cite{DVV17, DV17ar, VD3} and duality relations \cite{H11, Young}. Paley and Wiener's traditional perturbation result, which states that a sequence that is close enough to an orthonormal basis in a Hilbert space to naturally generate a basis, is where it all began \cite{H11}. Some Paley-Wiener type perturbation results for discrete frames in separable Hilbert spaces can be found in \cite{CK}. Perturbation of weaving frames for operators that have potential applications in distributed signal processing were studied in \cite{DV17ar}. Favier and Zalik presented stability conditions for frames of Gabor and wavelet systems in the paper \cite{FZUS}.
Very recently, authors of \cite{DivI} studied perturbation of  frames with Gabor and wavelet structure  in  the Lebesgue space $L^2(\mathbb{R})$ associated with  the Weyl-Heisenberg group and the extended affine group.
Different Paley-Wiener type perturbation results for various frame types in separable Hilbert spaces are accommodated in the texts by Gosson \cite{GossonII}, Heil \cite{H11}  and Young \cite{Young}. Motivated by above works, we derive new sufficient conditions for perturbation of Hilbert-Schmidt frames with explicit frame bounds. This paper investigates the stability of Hilbert-Schmidt frames under weaving perturbations, where finite or infinite number of frame elements are replaced or perturbed. The finite case models localized distortions encountered in practical scenarios such as packet loss in networks, sensor failure, or localized noise. We derive explicit quantitative conditions under which such weavings still yield Hilbert-Schmidt frames.


\subsection{Structure of the paper}
The work in the paper is organized as follows: In order to make the paper self-contained, necessary background on  Hilbert-Schmidt frames is given in Section \ref{sec2}. The main results are given in  Section \ref{sec3}. Firstly,  we present Theorem \ref{thm1} on stability under finite perturbations, supported by illustrative remark and example. Theorem \ref{thm2} establishes a quadratic perturbation bound that ensures frame stability under finite perturbations. Theorem \ref{th4} gives stability even when the perturbing sequence is not an HS-Bessel sequence.
Theorem \ref{thm3} extends the stability to the case where infinitely many frame elements are perturbed under the assumption that the perturbations decay rapidly. Example \ref{exa2} illustrates Theorem \ref{thm3}  by considering both polynomially and exponentially decaying perturbations. Finally, Theorem \ref{th5} shows that any weaving of sufficiently close Parseval Hilbert-Schmidt frames remains a Hilbert-Schmidt frame.

\subsection{Notation and Definitions}\label{sec2}
Symbol $I$ denotes a countable indexing set. We recall that for given sequences  $\{x_i\}_{i \in I}$ and $\{y_i\}_{i \in I}$ in $\mathcal{H}$ and a subset $\sigma \subset I$, the collection  $\{x_i\}_{i \in \sigma} \bigcup \{y_i\}_{i \in \sigma^{c}}$ is called a \emph{weaving} of $\{x_i\}_{i \in I}$ and $\{y_i\}_{i \in I}$, where $\sigma^{c}$ is the complement of $\sigma$. For frame conditions of weaving frames, we refer to \cite{DVV17} and references therein.  The collection of bounded linear operators from a normed space $\mathcal{X}$ into a normed space $\mathcal{Y}$ is denoted by  $\mathfrak{B}(\mathcal{X}, \mathcal{Y})$.  If $\mathcal{X}=\mathcal{Y}$, then we write $\mathfrak{B}(\mathcal{X}, \mathcal{Y})= \mathfrak{B}(\mathcal{X})$.  Let $\{e_i\}_{i \in I}$ be  an orthonormal basis of a separable Hilbert space $\mathcal{K}$. The Hilbert-Schimdt class, denoted by $\mathcal{C}_2(\mathcal{K})$, is the collection of all compact operators $T$ acting on $\mathcal{K}$ such that $\sum_{i \in I} \|Te_i\|^2 < \infty$. The space $\mathcal{C}_2(\mathcal{K})$ is a Banach space with respect to the norm $\|.\|_2$ defined as
\begin{align*}
\|T\|_2= \Big(\text{\textbf{trace}}(T^* T)\Big)^\frac{1}{2} =\big(\sum\limits_{i \in I} \| Te_i\|^2\big)^\frac{1}{2}, \ T \in \mathcal{C}_2(\mathcal{K}).
\end{align*}
 The space $\mathcal{C}_2(\mathcal{K})$ is a Hilbert space with respect to the inner product $[ .,. ]_{\textbf{tr}}$ defined as
\begin{align*}
[ T, S ]_{\textbf{tr}}= \text{\textbf{trace}}(S^* T)=\sum_{i \in I} \langle T e_i, S e_i \rangle , \ T, S \in \mathcal{C}_2(\mathcal{K}).
\end{align*}
For  $ x, y \in \mathcal{K}$, the operator $x \otimes y : \mathcal{K} \rightarrow \mathcal{K}$ is defined as
\begin{align*}
(x \otimes y )(u)= \langle u, y \rangle x, \ u \in \mathcal{K}.
\end{align*}
It can be easily verified that $x \otimes y \in \mathfrak{B}(\mathcal{K})$ with $\|x \otimes y\|= \|x\|\|y\|$. Moreover, $x \otimes y \in \mathcal{C}_2(\mathcal{K})$ with $\|x \otimes y\|_2= \|x\|\|y\|$ for $ x, y \in \mathcal{K}$. Also, $[ x \otimes y, u \otimes v ]_{\textbf{tr}}=\langle x, u \rangle \langle v, y \rangle$ for $ x, y, u, v \in \mathcal{K}$.

\begin{defn}\cite{SA}
A sequence $\{\Theta_i\}_{i \in I}$ of bounded linear operators from a Hilbert space $\mathcal{H}$ into the Hilbert-Schmidt class $\mathcal{C}_2(\mathcal{K})$ is called a Hilbert-Schmidt frame (HS-frame) for $\mathcal{H}$ with respect to $\mathcal{K}$ if there exist constants $A, B > 0$ such that
\begin{align}\label{eq3.2mf}
A \|x\|^2 \leq \sum_{i \in I} \|\Theta_i(x)\|_2^2 \leq B \|x\|^2 \quad \text{for all} \ x \in \mathcal{H}.
\end{align}
\end{defn}
The scalars $A$ and $B$ are called frame bounds of $\{\Theta_i\}_{i \in I}$. The HS-frame $\{\Theta_i\}_{i \in I}$ is called a Parseval HS-frame if it is possible to choose $A=B=1$.
If only the upper inequality in \eqref{eq3.2mf} is satisfied,   then we say that $\{\Theta_i\}_{i \in I}$  is  a Hilbert-Schmidt Bessel  (HS-Bessel) sequence with Bessel bound $B$.
In a very recent work, the authors in \cite{JVPDF1} studied Hilbert-Schmidt frames and Hilbert-Schmidt Riesz bases with respect to the tensor product of Hilbert spaces. A characterization of  Hilbert-Schmidt frames  with respect to the tensor product of Hilbert spaces can be found in \cite{JVPDF1}. The duality of  Hilbert-Schmidt frames  with respect to the tensor product of Hilbert spaces is also discussed in \cite{JVPDF1}.  In \cite{J}, the first author of this paper characterized Hilbert-Schmidt frames for the Hilbert-Schmidt class $\mathcal{C}_2(\mathcal{K})$ in terms of  surjectivity of a  bounded linear operator on $\mathcal{C}_2(\mathcal{K})$. Hilbert-Schmidt frames for separable Hilbert spaces, where the lower frame condition is controlled by a bounded linear operator were studied by the authors in \cite{JV23}.


\section{Main  Results}\label{sec3}

The following results quantify how small perturbations affect HS-frames under finite replacements. These findings are motivated by applications in distributed systems where only a subset of frame elements may be altered. The first result below establishes sufficient conditions under which a sequence formed by finite replacements in an HS-frame retains the frame property, with explicit dependence on the perturbation size and the cardinality of the modified set.
\begin{thm}\label{thm1}
Let  $\{G_i\}_{i \in I}$ be an HS-frame for $\mathcal{H}$ with respect to $\mathcal{K}$ with frame bounds $A_G, B_G$, and $\{F_i\}_{i \in I}$ be an HS-Bessel sequence for
$\mathcal{H}$ with respect to $\mathcal{K}$ with Bessel bound $B_F$. Fix any subset $\sigma \subseteq I$ with finite cardinality $|\sigma| \leq N < \infty$ and define the weaving  $\{H_i\}_{i \in I}$ as
\begin{align*}
H_i = \begin{cases}
F_i, & i \in \sigma, \\
G_i, & i \in \sigma^c.
\end{cases}
\end{align*}
Assume that  $\|F_i - G_i\| \leq \epsilon$ for all $i \in I$, where
\begin{align*}
0 <\epsilon < \frac{A_G}{2 N\max(\sqrt{B_F}, \sqrt{B_G})}.
\end{align*}
Then, $\{H_i\}_{i \in I}$ is an HS-frame for $\mathcal{H}$ with respect to $\mathcal{K}$ with frame bounds $\big[A_G - 2 N \epsilon \max(\sqrt{B_F}, \sqrt{B_G}) \big]$ and $\big[B_G + 2 N \epsilon \max(\sqrt{B_F}, \sqrt{B_G}) \big]$.
\end{thm}
\begin{proof}
For any  $x \in \mathcal{H}$, we compute
\begin{align*}
&\Big| \sum_{i \in I} \|H_i(x)\|_2^2 - \sum_{i \in I} \|G_i(x)\|_2^2 \Big| \\
&\Big|  \sum_{i \in \sigma} \|F_i(x)\|_2^2 + \sum_{i \in \sigma^c} \|G_i(x)\|_2^2 - \sum_{i \in I} \|G_i(x)\|_2^2 \Big| \\
&= \Big| \sum_{i \in \sigma} \big( \|F_i(x)\|_2^2 - \|G_i(x)\|_2^2 \big) \Big|\\
& \leq   \sum_{i \in \sigma} \Big|  \|F_i(x)\|_2^2 - \|G_i(x)\|_2^2  \Big|\\
&=  \sum_{i \in \sigma} \Big| [F_i(x), F_i(x) ]_{\text{tr}}-[G_i(x), G_i(x) ]_{\text{tr}}\Big|\\
&=  \sum_{i \in \sigma} \Big| [F_i(x), F_i(x) ]_{\text{tr}}-[G_i(x), F_i(x) ]_{\text{tr}}+[G_i(x), F_i(x) ]_{\text{tr}}-[G_i(x), G_i(x) ]_{\text{tr}}\Big|\\
&=  \sum_{i \in \sigma} \Big| [F_i(x)-G_i(x), F_i(x) ]_{\text{tr}}+[G_i(x), F_i(x)-G_i(x)]_{\text{tr}}\Big|\\
&\leq  \sum_{i \in \sigma} \Big(\big| [F_i(x)-G_i(x), F_i(x) ]_{\text{tr}}\big|+\big|[G_i(x), F_i(x)-G_i(x)]_{\text{tr}}\big|\Big)\\
&\leq  \sum_{i \in \sigma} \Big(\| F_i(x)-G_i(x)\|_2 \|F_i(x) \|_2+\|G_i(x)\|_2 \|F_i(x)-G_i(x)\|_2\Big)\\
&=  \sum_{i \in \sigma} \| F_i(x)-G_i(x)\|_2 \big(\|F_i(x) \|_2+\|G_i(x)\|_2\big)\\
&\leq  \sum_{i \in \sigma} \|F_i - G_i\|\|x\| \big(\sqrt{B_F} \|x\|+\sqrt{B_G} \|x\|\big)\\
&\leq  \sum_{i \in \sigma} 2\epsilon  \max(\sqrt{B_F}, \sqrt{B_G}) \|x\|^2 \\
&\leq   2 N \epsilon  \max(\sqrt{B_F}, \sqrt{B_G}) \|x\|^2.
\end{align*}
Using this, we get
\begin{align*}
\sum_{i \in I} \|H_i(x)\|_2^2 &=\Big|\sum_{i \in I}( \|H_i(x)\|_2^2 - \|G_i(x)\|_2^2 + \|G_i(x)\|_2^2) \Big|\\
&\geq \sum_{i \in I} \|G_i(x)\|_2^2-\Big| \sum_{i \in I} \|H_i(x)\|_2^2 - \sum_{i \in I} \|G_i(x)\|_2^2 \Big|\\
&\geq A_G \|x\|^2- 2 N \epsilon \max(\sqrt{B_F}, \sqrt{B_G}) \|x\|^2\\
&=\big[A_G - 2 N \epsilon \max(\sqrt{B_F}, \sqrt{B_G}) \big] \|x\|^2, \ x \in \mathcal{H}.
\end{align*}
This proves the lower frame inequality. The upper frame inequality follows by similar estimates. That is, for any  $x \in \mathcal{H}$, we have
\begin{align*}
\sum_{i \in I} \|H_i(x)\|_2^2 &=\Big|\sum_{i \in I}( \|H_i(x)\|_2^2 - \|G_i(x)\|_2^2 + \|G_i(x)\|_2^2) \Big|\\
&\leq \sum_{i \in I} \|G_i(x)\|_2^2+\Big| \sum_{i \in I} \|H_i(x)\|_2^2 - \sum_{i \in I} \|G_i(x)\|_2^2 \Big|\\
&\leq B_G \|x\|^2+2 N \epsilon \max(\sqrt{B_F}, \sqrt{B_G}) \|x\|^2\\
&=\big[B_G + 2 N \epsilon \max(\sqrt{B_F}, \sqrt{B_G}) \big] \|x\|^2.
\end{align*}
The proof is complete.
\end{proof}

\begin{rem}
The above result critically depends on the assumption that the subset $\sigma$ where the operators are replaced is finite with $|\sigma| \leq N$. This ensures the perturbation caused by switching from $\{G_i\}_{i \in I}$ to $\{F_i\}_{i \in I}$ accumulates over finitely many indices, keeping the total perturbation small. The complement $\sigma^c = I \setminus \sigma$ may be infinite without affecting the stability because on $\sigma^c$ the operators remain unchanged as $\{G_i\}$. No additional perturbation is introduced there, so the infinite size of $\sigma^c$ poses no problem.
This models practical scenarios where only finitely many operators differ in the weaving, while the rest come from a stable HS-frame.
\end{rem}

The following example demonstrates how the finite weaving stability result applies concretely to a scenario with a small number of local perturbations in HS-frames.
\begin{exa}
Consider the index set $I = \mathbb{N}$ and fix $N=3$. Let $\{G_i\}_{i \in I}$ be an HS-frame with the frame bounds $A_G = 1, B_G = 2$, and $\{F_i\}_{i \in I}$ be an HS-Bessel sequence with Bessel bound $B_F=2$. Choose perturbation parameter $\epsilon = 0.1$. Then
\begin{align*}
2 N \epsilon \max(\sqrt{B_F}, \sqrt{B_G})  = 2 \times 3 \times 0.1 \times \sqrt{2} \approx 0.8485 < A_G = 1,
\end{align*}
which by Theorem \ref{thm1} ensures that the weaving formed by choosing up to $3$ operators from $\{F_i\}_{i \in I}$ and the rest from $\{G_i\}_{i \in I}$ is still an HS-frame.
\end{exa}

We now present a complementary result that uses a quadratic perturbation bound to establish frame stability. This approach offers an alternative estimate for the lower frame bound, particularly useful when the perturbation size or the number of modified elements is relatively small.
\begin{thm}\label{thm2}
Let  $\{G_i\}_{i \in I}$ be an HS-frame for $\mathcal{H}$ with respect to $\mathcal{K}$ with frame bounds $A_G$, $B_G$, and $\{F_i\}_{i \in I}$ be an HS-Bessel sequence for
$\mathcal{H}$ with respect to $\mathcal{K}$ with Bessel bound $B_F$. Fix any subset $\sigma \subseteq I$ with finite cardinality $|\sigma| \leq N < \infty$ and define the weaving  $\{H_i\}_{i \in I}$ as
\begin{align*}
H_i = \begin{cases}
F_i, & i \in \sigma, \\
G_i, & i \in \sigma^c.
\end{cases}
\end{align*}
Assume that  $\|F_i - G_i\| \leq \epsilon$ for all $i \in I$, for $\epsilon>0$ that satisfies $N\epsilon^2 +2\epsilon\sqrt{N} \sqrt{B_G}<A_G$.
Then,  $\{H_i\}_{i \in I}$ is an HS-frame for $\mathcal{H}$ with frame bounds $\big[A_G-N\epsilon^2 -2\epsilon\sqrt{N} \sqrt{B_G}\big]$ and $(B_F + B_G)$.
\end{thm}
\begin{proof}
It is straightforward to obtain the upper frame condition for the weaving sequence $\{H_i\}_{i \in I}$. Precisely, for any $x \in \mathcal{H}$, we have
\begin{align*}
\sum_{i \in I} \|H_i(x)\|_2^2 =\sum_{i \in \sigma} \|F_i(x)\|_2^2 +\sum_{i \in \sigma^c} \|G_i(x)\|_2^2 \leq \sum_{i \in I} \|F_i(x)\|_2^2+\sum_{i \in I} \|G_i(x)\|_2^2 \leq (B_F + B_G)\|x\|^2.
\end{align*}
Next, to prove the lower frame condition, for each $x \in \mathcal{H}$, we compute
\begin{align}\label{eq2.1}
\sum_{i \in I} \|H_i(x)\|_2^2 &=\sum_{i \in I} \|H_i(x)-G_i(x)+G_i(x)\|_2^2\notag\\
& =\sum_{i \in I}[ H_i(x)-G_i(x)+G_i(x), H_i(x)-G_i(x)+G_i(x)]_{\text{tr}}\notag\\
&= \sum_{i \in I} \|H_i(x)-G_i(x)\|_2^2 + 2 \text{Re} \sum_{i \in I} [ H_i(x)-G_i(x), G_i(x) ]_{\text{tr}} + \sum_{i \in I} \|G_i(x)\|_2^2\notag\\
&= \sum_{i \in \sigma} \|F_i(x)-G_i(x)\|_2^2 + 2 \text{Re} \sum_{i \in \sigma} [ F_i(x)-G_i(x), G_i(x) ]_{\text{tr}} + \sum_{i \in I} \|G_i(x)\|_2^2.
\end{align}
We see that
\begin{align}\label{eq2.2}
&2 \text{Re} \sum_{i \in \sigma} [ F_i(x)-G_i(x), G_i(x) ]_{\text{tr}} \notag\\
& \leq 2 \Big|\sum_{i \in \sigma} [ F_i(x)-G_i(x), G_i(x) ]_{\text{tr}} \Big|\notag\\
& \leq 2 \sum_{i \in \sigma}| [ F_i(x)-G_i(x), G_i(x) ]_{\text{tr}}| \notag\\
& \leq 2 \sum_{i \in \sigma}\| F_i(x)-G_i(x)\|_2\| G_i(x)\|_2 \notag\\
&\leq 2 \big(\sum_{i \in \sigma}\| F_i(x)-G_i(x)\|_2^2\big)^{1/2}\big(\sum_{i \in \sigma}\| G_i(x)\|_2^2\big)^{1/2}\notag\\
& \leq 2 \big(\sum_{i \in \sigma}\| F_i-G_i\|^2\|x\|^2\big)^{1/2}\big(\sum_{i \in I}\| G_i(x)\|_2^2\big)^{1/2}\notag\\
&\leq 2 \big(\sum_{i \in \sigma}\epsilon^2\|x\|^2\big)^{1/2}(B_G\|x\|^2)^{1/2}\notag\\
& \leq 2\epsilon\sqrt{N} \sqrt{B_G}\|x\|^2, \, x \in \mathcal{H},
\end{align}
and
\begin{align}
&\sum_{i \in \sigma} \|F_i(x)-G_i(x)\|_2^2 \leq \sum_{i \in \sigma} \|F_i-G_i\|^2 \|x\|^2  \leq \sum_{i \in \sigma} \epsilon^2 \|x\|^2\leq N\epsilon^2 \|x\|^2, \, x \in \mathcal{H}. \label{eq2.3}
\end{align}
Hence, rewriting equation \eqref{eq2.1} and using equations \eqref{eq2.2} and \eqref{eq2.3}, we get
\begin{align*}
\sum_{i \in I} \|H_i(x)\|_2^2 &=\Big|\sum_{i \in \sigma} \|F_i(x)-G_i(x)\|_2^2 + 2 \text{Re} \sum_{i \in \sigma} [ F_i(x)-G_i(x), G_i(x) ]_{\text{tr}} + \sum_{i \in I} \|G_i(x)\|_2^2\Big|\\
&\geq \Big|\sum_{i \in I} \|G_i(x)\|_2^2\Big|-\Big|\sum_{i \in \sigma} \|F_i(x)-G_i(x)\|_2^2 + 2 \text{Re} \sum_{i \in \sigma} [ F_i(x)-G_i(x), G_i(x) ]_{\text{tr}} \Big|\\
& \geq \sum_{i \in I} \|G_i(x)\|_2^2 - \Big|\sum_{i \in \sigma} \|F_i(x)-G_i(x)\|_2^2\Big|-\Big|2 \text{Re} \sum_{i \in \sigma} [ F_i(x)-G_i(x), G_i(x) ]_{\text{tr}} \Big|\\
& \geq A_G\|x\|^2 -N\epsilon^2 \|x\|^2-2\epsilon\sqrt{N} \sqrt{B_G}\|x\|^2\\
&=\big[A_G-N\epsilon^2 -2\epsilon\sqrt{N} \sqrt{B_G}\big]\|x\|^2, \ x \in \mathcal{H}.
\end{align*}
This completes the proof.
\end{proof}

Next, we give a perturbation result that involves a sequence $\{F_i\}_{i \in I}$, which is used over a finite subset $\sigma \subset I$, despite not being an HS-Bessel sequence. The only assumption needed is that the perturbation from $\{G_i\}_{i \in I}$ is controlled locally over finite subsets. This local nature of the condition allows the result to be applied even when $\{F_i\}_{i \in I}$ is defined only partially, or behaves poorly outside a finite set, making the result particularly useful in finite-dimensional approximations and localized perturbation scenarios.
\begin{thm}\label{th4}
Let  $\{G_i\}_{i \in I}$ be an HS-frame for $\mathcal{H}$ with respect to $\mathcal{K}$ with frame bounds $A_G, B_G$, and $\{F_i\}_{i \in I}$ be a sequence of bounded linear operators from  $\mathcal{H}$ into $\mathcal{C}_2(\mathcal{K})$. Suppose $\delta \in (0,1)$ be a constant such that for every finite subset $J \subset I$,
\begin{align*}
\sum_{i \in J} \| (F_i - G_i)(x) \|_2^2 \leq \delta \sum_{i \in J} \| G_i(x) \|_2^2 \quad \text{for all} \ x \in \mathcal{H}.
\end{align*}
Then, for every finite subset $\sigma \subset I$, the weaving $\{H_i\}_{i \in I}$ defined by
\begin{align*}
H_i = \begin{cases}
F_i, & i \in \sigma, \\
G_i, & i \in \sigma^c
\end{cases}
\end{align*}
is an HS-frame for $\mathcal{H}$ with respect to $\mathcal{K}$ with frame bounds $\big[(1 - \sqrt{\delta})^2 A_G\big] \ \text{and} \ \big[2(\delta +1) B_G\big]$.
\end{thm}
\begin{proof}
From the hypothesis,  for any $i \in I$, it follows that
\begin{align*}
\| (F_i - G_i)(x) \|_2 \leq \sqrt{\delta} \| G_i(x) \|_2 \quad \text{for all} \ x \in \mathcal{H}
\intertext{which gives}
 \| F_i(x) \|_2 \geq  \| G_i(x) \|_2- \| (F_i - G_i)(x) \|_2 \geq (1-\sqrt{\delta}) \| G_i(x) \|_2 .
\end{align*}
Using this, we compute
\begin{align*}
\sum_{i \in I}\| H_i(x) \|_2^2&=\sum_{i \in \sigma}\| F_i(x) \|_2^2 +\sum_{i \in \sigma^c}\| G_i(x) \|_2^2\\
&\geq \sum_{i \in \sigma} (1-\sqrt{\delta})^2 \| G_i(x) \|_2 ^2 +\sum_{i \in \sigma^c}\| G_i(x) \|_2^2\\
& \geq (1-\sqrt{\delta})^2 \sum_{i \in I}\| G_i(x) \|_2^2\\
& \geq (1-\sqrt{\delta})^2 A_G \|x\|^2, \ x \in \mathcal{H}.
\end{align*}
This gives the lower frame inequality for $\{H_i\}_{i \in I}$. Next, we compute
\begin{align*}
\sum_{i \in I}\| H_i(x) \|_2^2&= \sum_{i \in \sigma}\| F_i(x) \|_2^2 +\sum_{i \in \sigma^c}\| G_i(x) \|_2^2\\
&= \sum_{i \in \sigma}\| (F_i - G_i)(x) +G_i(x) \|_2^2 +\sum_{i \in \sigma^c}\| G_i(x) \|_2^2\\
&\leq 2\sum_{i \in \sigma}\| (F_i - G_i)(x)\|_2^2 + 2\sum_{i \in \sigma}\|G_i(x) \|_2^2 +\sum_{i \in \sigma^c}\| G_i(x) \|_2^2\\
&\leq 2\delta \sum_{i \in \sigma} \| G_i(x) \|_2^2 + 2\sum_{i \in \sigma}\|G_i(x) \|_2^2 +\sum_{i \in \sigma^c}\| G_i(x) \|_2^2\\
&= 2(\delta +1) \sum_{i \in \sigma} \| G_i(x) \|_2^2 +\sum_{i \in \sigma^c}\| G_i(x) \|_2^2\\
&\leq  2(\delta +1) \sum_{i \in I} \| G_i(x) \|_2^2 \\
&\leq  2(\delta +1) B_G \|x\|^2,  \ x \in \mathcal{H}.
\end{align*}
Hence, $\{H_i\}_{i \in I}$ is an HS-frame for $\mathcal{H}$ with respect to $\mathcal{K}$ with required frame bounds.
\end{proof}

\subsection{Special Relevance in Finite-Dimensional Settings}

When both $\mathcal{H}$ and $\mathcal{K}$ are finite-dimensional spaces, the results acquire special relevance. Here, Hilbert-Schmidt and operator norms are equivalent, and computations become more tractable. Moreover, the finite index sets naturally match the assumption of finitely supported perturbations, making these results especially applicable to finite sensor networks, discrete signal systems, and matrix data processing. The explicit bounds derived offer concrete criteria for ensuring robust reconstruction under local distortions.

\subsection{Extension to Infinite Weaving with Decaying Perturbations}
While finite perturbations naturally model localized failure, many real-world systems undergo gradual degradation, where frame elements are corrupted in a decaying manner. In such cases, control over the aggregate perturbation, rather than pointwise control, suffices. The next result generalizes Theorem \ref{thm1} by allowing infinitely many perturbed elements, provided the perturbations decay sufficiently fast to ensure square-summability.
\begin{thm}\label{thm3}
Let $\{G_i\}_{i \in I}$ be an HS-frame for $\mathcal{H}$  with respect to $\mathcal{K}$  with frame bounds $A_G, B_G > 0$. Let $\{F_i\}_{i \in I}$ be an HS-Bessel sequence for $\mathcal{H}$  with respect to $\mathcal{K}$ with Bessel bound $B_F > 0$. Suppose
\begin{align*}
\|F_i - G_i\| &\leq \frac{\epsilon}{w_i}, \quad \text{for all } i \in I,
\end{align*}
where $\epsilon > 0$ and the sequence of positive weights $\{w_i\}_{i \in I}$ satisfy
\begin{align*}
\sum_{i \in I} \frac{1}{w_i^2} < \infty  \quad \quad \text{and} \quad \quad  \epsilon \Big(\sum_{i \in I} \frac{1}{w_i^2}\Big)^{1/2}\big(\sqrt{B_F}+ \sqrt{B_G}\big) < A_G.
\end{align*}
Then, for any subset $\sigma \subset I$, the weaving $\{H_i\}_{i \in I}$ given by
\begin{align}\label{defse1}
    H_i &= \begin{cases}
    F_i, & i \in \sigma,\\
    G_i, & i \in \sigma^c
    \end{cases}
\end{align}
forms an HS-frame for $H$ with frame bounds
$\big(A_G - \epsilon \big(\sum_{i \in I} \frac{1}{w_i^2}\big)^{1/2}\big(\sqrt{B_F}+ \sqrt{B_G}\big)\big)$ and $(B_F + B_G)$.
\end{thm}
\begin{proof}
We only need to prove the lower frame inequality. Let  $\{H_i\}_{i \in I}$ be given by \eqref{defse1}. Using the initial calculations as done in the proof of Theorem \ref{thm1}, for every $x \in \mathcal{H}$, we get
\begin{align*}
&\Big| \sum_{i \in I} \|H_i(x)\|_2^2 - \sum_{i \in I} \|G_i(x)\|_2^2 \Big| \\
&\leq  \sum_{i \in \sigma} \| F_i(x)-G_i(x)\|_2 \big(\|F_i(x) \|_2+\|G_i(x)\|_2\big)\\
&\leq \Big(\sum_{i \in \sigma} \| F_i(x)-G_i(x)\|_2^2\Big)^{1/2}\Big(\sum_{i \in \sigma} \| F_i(x)\|_2^2\Big)^{1/2} +\Big(\sum_{i \in \sigma} \| F_i(x)-G_i(x)\|_2^2\Big)^{1/2}\Big(\sum_{i \in \sigma} \| G_i(x)\|_2^2\Big)^{1/2}\\
&\leq \Big(\sum_{i \in \sigma} \| F_i-G_i\|^2\|x\|^2\Big)^{1/2} \Big[\Big(\sum_{i \in \sigma} \| F_i(x)\|_2^2\Big)^{1/2} + \Big(\sum_{i \in \sigma} \| G_i(x)\|_2^2\Big)^{1/2} \Big]\\
&\leq \Big(\sum_{i \in \sigma} \frac{\epsilon^2}{w_i^2} \|x\|^2\Big)^{1/2} \Big[\Big(\sum_{i \in I} \| F_i(x)\|_2^2\Big)^{1/2} + \Big(\sum_{i \in I} \| G_i(x)\|_2^2\Big)^{1/2} \Big]\\
&\leq \Big(\sum_{i \in I} \frac{\epsilon^2}{w_i^2} \|x\|^2\Big)^{1/2} \big(\sqrt{B_F} \|x\|+ \sqrt{B_G} \|x\|\big)\\
&= \epsilon \Big(\sum_{i \in I} \frac{1}{w_i^2}\Big)^{1/2}\big(\sqrt{B_F}+ \sqrt{B_G}\big) \|x\|^2, \ x \in \mathcal{H}.\\
\end{align*}
Using this, we get
\begin{align*}
\sum_{i \in I} \|H_i(x)\|_2^2 &=\Big|\sum_{i \in I}( \|H_i(x)\|_2^2 - \|G_i(x)\|_2^2 + \|G_i(x)\|_2^2) \Big|\\
&\geq \sum_{i \in I} \|G_i(x)\|_2^2-\Big| \sum_{i \in I} \|H_i(x)\|_2^2 - \sum_{i \in I} \|G_i(x)\|_2^2 \Big|\\
&\geq A_G \|x\|^2- \epsilon \Big(\sum_{i \in I} \frac{1}{w_i^2}\Big)^{1/2}\big(\sqrt{B_F}+ \sqrt{B_G}\big) \|x\|^2\\
&=\Big[A_G - \epsilon \Big(\sum_{i \in I} \frac{1}{w_i^2}\Big)^{1/2}\big(\sqrt{B_F}+ \sqrt{B_G}\big)\Big] \|x\|^2, \ x \in \mathcal{H}.
\end{align*}
This concludes the proof.
\end{proof}

The following example illustrates the applicability of Theorem \ref{thm3} to practical scenarios where infinitely many operators are perturbed, but the perturbations decay sufficiently fast to preserve the frame structure. Unlike previous results which are restricted to finitely supported perturbations, these cases show how infinite weaving remains stable under controlled decay conditions. This framework captures realistic models in which frame elements degrade with decreasing severity for instance, over time, spatial distance, or system iterations, making the theory relevant for adaptive signal systems, sensor arrays, and learning networks. In Example \ref{exa2}(a), we consider polynomial decay in the operator perturbations, allowing the difference $\|F_i-G_i\|$ to diminish like $\frac{1}{i^p}$ for some $p>\frac{1}{2}$. This ensures square-summability of the decay weights and leads to a stable weaving HS-frame regardless of the choice of subset $\sigma$. In Example \ref{exa2}(b), we further strengthen the result by considering exponential decay, where perturbations diminish at a rate $e^{-ci}$. Such exponential models are natural in systems with fast convergence, feedback stabilization, or time-decaying error dynamics.

\begin{exa} \label{exa2}
Let $\mathcal{H}=\mathcal{K}=\ell^2(\mathbb{N})$. Consider the canonical orthonormal basis $\{e_i\}_{i \in \mathbb{N}}$ of $\ell^2(\mathbb{N})$. Define $G_i :\ell^2(\mathbb{N}) \to \mathcal{C}_2(\ell^2(\mathbb{N}))$ by
\begin{align*}
    G_i(x) &= \begin{cases}
    \langle x, e_1 \rangle e_1 \otimes e_1, & i=1\\
    \langle x, e_{i-1} \rangle e_{i} \otimes e_{i}, & i \geq 2.
    \end{cases}
\end{align*}
Then,
\begin{align*}
\|x\|^2 \leq \sum_{i \in  \mathbb{N}}\|G_i(x)\|_2^2&= \| \langle x, e_1 \rangle e_1 \otimes e_1\|_2^2+\sum_{i=2}^\infty \| \langle x, e_{i-1} \rangle e_{i} \otimes e_{i}\|_2^2\\
&=| \langle x, e_1 \rangle |^2+\sum_{i =2}^\infty | \langle x, e_{i-1} \rangle|^2\\
&=| \langle x, e_1 \rangle |^2+\sum_{i =1}^\infty | \langle x, e_{i} \rangle|^2\\
& \leq 2 \|x\|^2, \ x \in \ell^2(\mathbb{N}).
\end{align*}
This shows that $\{G_i\}_{i \in \mathbb{N}}$ is an HS-frame for $\ell^2(\mathbb{N})$ with frame bounds $A_G = 1,  B_G = 2$.\\
Let $\{F_i\}_{i \in \mathbb{N}}$ be a perturbed sequence defined as $F_i = G_i + E_i$, with $\|E_i\| \leq \frac{\epsilon}{w_i}$, where $\epsilon > 0$.
\begin{enumerate}[$(a)$]
\item Define weights $w_i = i^p$ for some $p > \frac{1}{2}$.
Then the perturbations satisfy $\|F_i - G_i\| \leq \frac{\epsilon}{i^p}$, and the perturbations decay polynomially. Then,
\begin{align*}
\sum_{i \in \mathbb{N}} \frac{1}{w_i^2} = \sum_{i \in \mathbb{N}} \frac{1}{i^{2p}} < \infty.
\end{align*}
We see that for any $x \in \ell^2(\mathbb{N})$, we have
\begin{align*}
\sum_{i \in  \mathbb{N}} \|F_i(x)\|_2^2&=\sum_{i \in  \mathbb{N}} \|G_i (x)+ E_i(x)\|_2^2 \\
&\leq 2\sum_{i \in  \mathbb{N}} \|G_i (x)\|_2^2+2\sum_{i \in  \mathbb{N}} \|E_i (x)\|_2^2 \\
&\leq 2 B_G\|x\|^2+2\sum_{i \in  \mathbb{N}} \|E_i\|^2 \|x\|^2\\
&\leq 2 B_G\|x\|^2+2\epsilon^2 \sum_{i \in  \mathbb{N}} \frac{1}{i^{2p}} \|x\|^2\\
&=2\Big(B_G+\epsilon^2 \sum_{i \in  \mathbb{N}} \frac{1}{i^{2p}}\Big)\|x\|^2\\
&=2\big(B_G+\epsilon^2 \zeta(2p)\big)\|x\|^2,
\end{align*}
which shows that $\{F_i\}_{i \in \mathbb{N}}$ is an HS-Bessel sequence with bound $B_F=2\big(B_G+\epsilon^2 \zeta(2p)\big)$, here $\zeta(2p)$ is the Riemann zeta function at $2p$.
Choose $\epsilon > 0$ small enough to satisfy
\begin{align*}
\epsilon \Big(\sum_{i \in \mathbb{N}} \frac{1}{w_i^2}\Big)^{1/2}\big(\sqrt{B_F}+ \sqrt{B_G}\big)=\epsilon \sqrt{\zeta(2p)}\big(\sqrt{2\big(2+\epsilon^2 \zeta(2p)\big)}+ \sqrt{2}\big) < 1=A_G.
\end{align*}
By Theorem \ref{thm3}, for any subset $\sigma \subset \mathbb{N}$, the weaving $\{F_i\}_{i \in \sigma} \cup \{G_i\}_{i \in \sigma^c}$
is an HS-frame for $\ell^2(\mathbb{N})$.

\item Let weights be $w_i = e^{c i}, c \in \mathbb{R}$. Then, the perturbations satisfy $\|F_i - G_i\| \leq \frac{\epsilon}{e^{c i}}$,
and
\begin{align*}
\sum_{i \in \mathbb{N}} \frac{1}{w_i^2} = \sum_{i \in \mathbb{N}} e^{-2 c i} = \frac{e^{-2 c}}{1 - e^{-2 c}} < \infty.
\end{align*}
As done in part $(a)$, $\{F_i\}_{i \in \mathbb{N}}$ is an HS-Bessel sequence with bound $B_F=2\big(B_G+ \frac{\epsilon^2e^{-2 c}}{1 - e^{-2 c}} \big)$.
Choose $\epsilon > 0$ small enough so that
\begin{align*}
\epsilon \Big(\sum_{i \in \mathbb{N}} \frac{1}{w_i^2}\Big)^{1/2}\big(\sqrt{B_F}+ \sqrt{B_G}\big)=\epsilon \sqrt{\frac{e^{-2 c}}{1 - e^{-2 c}}}\left(\sqrt{2\Big(2+ \frac{\epsilon^2e^{-2 c}}{1 - e^{-2 c}}\Big)}+ \sqrt{2}\right) < 1=A_G.
\end{align*}
By Theorem \ref{thm3}, for any subset $\sigma \subset \mathbb{N}$, the weaving $\{F_i\}_{i \in \sigma} \cup \{G_i\}_{i \in \sigma^c}$ is an HS-frame for $\ell^2(\mathbb{N})$.
\end{enumerate}
\end{exa}

\subsection{Weaving of Parseval HS-frames}
The following result shows that if two Parseval HS-frames are sufficiently close in the inner product sense, then every weaving forms an HS-frame.
The closeness is quantified by a uniform lower bound on the real part of the inner product.
\begin{thm}\label{th5}
Let $\{G_i\}_{i \in I}$ and $\{F_i\}_{i \in I}$ be Parseval HS-frames for $\mathcal{H}$ with respect to $\mathcal{K}$. Suppose there exists a constant $\delta \in (0, 1/8)$  such that
\begin{align}\label{eqa1}
\text{Re} \sum_{i \in I} [ F_i(x), G_i(x) ]_{\text{tr}} \geq (1 - \delta) \|x\|^2 \quad \text{for all } x \in \mathcal{H}.
\end{align}
Then, for any $\sigma \subset I$, the weaving $\{H_i\}_{i \in I}$ given by
\begin{align*}
    H_i &= \begin{cases}
    F_i, & i \in \sigma,\\
    G_i, & i \in \sigma^c
    \end{cases}
\end{align*}
is an HS-frame for $\mathcal{H}$ with respect to $\mathcal{K}$ with frame bounds $(1 - 2\sqrt{2\delta})$ and $(1 + 2\sqrt{2\delta})$.
\end{thm}
\begin{proof}
First, we observe that for any $x \in \mathcal{H}$, we have
\begin{align}\label{eq1}
\sum_{i \in I} \|H_i(x)\|_2^2
&= \sum_{i \in \sigma} \|F_i(x)\|_2^2 + \sum_{i \in \sigma^c} \|G_i(x)\|_2^2 \notag \\
&= \sum_{i \in \sigma} \|F_i(x)\|_2^2 + \sum_{i \in \sigma^c} \|G_i(x)\|_2^2+\sum_{i \in \sigma} \|G_i(x)\|_2^2-\sum_{i \in \sigma} \|G_i(x)\|_2^2  \notag \\
&= \sum_{i \in \sigma} \|F_i(x)\|_2^2 +\sum_{i \in I} \|G_i(x)\|_2^2-\sum_{i \in \sigma} \|G_i(x)\|_2^2  \notag \\
&=\|x\|^2+\sum_{i \in \sigma} \|F_i(x)\|_2^2 -\sum_{i \in \sigma} \|G_i(x)\|_2^2,  \notag \\
\intertext{so that}
&\Big|\sum_{i \in I} \|H_i(x)\|_2^2 - \|x\|^2\Big|=\Big|\sum_{i \in \sigma} \|F_i(x)\|_2^2 -\sum_{i \in \sigma} \|G_i(x)\|_2^2\Big|.
\end{align}
Now,
\begin{align*}
&\Big|\sum_{i \in \sigma} \|F_i(x)\|_2^2 -\sum_{i \in \sigma} \|G_i(x)\|_2^2\Big| \\
& \leq  \sum_{i \in \sigma} \| F_i(x)-G_i(x)\|_2 \big(\|F_i(x) \|_2+\|G_i(x)\|_2\big)\\
& \leq  \sum_{i \in I} \| F_i(x)-G_i(x)\|_2 \big(\|F_i(x) \|_2+\|G_i(x)\|_2\big)\\
& =  \sum_{i \in I} \| F_i(x)-G_i(x)\|_2 \|F_i(x) \|_2+\sum_{i \in I} \| F_i(x)-G_i(x)\|_2\|G_i(x)\|_2\\
&\leq  \big(\sum_{i \in I}\| F_i(x)-G_i(x)\|_2^2\big)^{1/2}\big(\sum_{i \in I}\| F_i(x)\|_2^2\big)^{1/2}+ \big(\sum_{i \in I}\| F_i(x)-G_i(x)\|_2^2\big)^{1/2}\big(\sum_{i \in I}\| G_i(x)\|_2^2\big)^{1/2} \\
&= 2 \|x\|\big(\sum_{i \in I}\| F_i(x)-G_i(x)\|_2^2\big)^{1/2}\\
&= 2 \|x\|\big(\sum_{i \in I}[F_i(x)-G_i(x), F_i(x)-G_i(x)]_{\text{tr}}\big)^{1/2}\\
&=2 \|x\|\big(\sum_{i \in I}\|F_i(x)\|_2^2 + \sum_{i \in I}\|G_i(x)\|_2^2- 2\text{Re} \sum_{i \in I} [ F_i(x), G_i(x) ]_{\text{tr}}\big)^{1/2}\\
&=2 \|x\|\big(2\|x\|^2- 2\text{Re} \sum_{i \in I} [ F_i(x), G_i(x) ]_{\text{tr}}\big)^{1/2}\\
& \leq 2 \|x\|\big(2\|x\|^2- 2(1 - \delta) \|x\|^2\big)^{1/2}\\
& = 2 \|x\|\big(2\delta\|x\|^2\big)^{1/2}\\
&= 2 \sqrt{2 \delta} \|x\|^2,\\
\intertext{which in equation \eqref{eq1} implies that}
&\Big|\sum_{i \in I} \|H_i(x)\|_2^2 - \|x\|^2\Big| \leq 2 \sqrt{2 \delta} \|x\|^2\\
\intertext{and hence}
&(1-2 \sqrt{2 \delta}) \|x\|^2 \leq \sum_{i \in I} \|H_i(x)\|_2^2  \leq (1+2 \sqrt{2 \delta}) \|x\|^2 \ \text{for all} \ x \in \mathcal{H}.
\end{align*}
The proof is complete.
\end{proof}

The following example illustrates Theorem \ref{th5} from a perturbation perspective. Here, $\{F_i\}_{i \in I}$ is constructed as a small perturbation of a Parseval HS-frame $\{G_i\}_{i \in I}$, so that both remain Parseval HS-frames and are sufficiently close in the inner product sense as required by condition \eqref{eqa1}. As a consequence, every weaving of $\{F_i\}_{i \in I}$ and $\{G_i\}_{i \in I}$ forms an HS-frame with explicitly controlled frame bounds.
\begin{exa}\label{th5exa}
Let $\mathcal{H} = \mathcal{K} = \ell^2(\mathbb{N})$, and let $\{e_i\}_{i \in \mathbb{N}}$ be the canonical orthonormal basis of $\ell^2(\mathbb{N})$. Define $G_i:\ell^2(\mathbb{N}) \to  \mathcal{C}_2(\ell^2(\mathbb{N}))$ as
\begin{align*}
G_i(x)=\langle x, e_i \rangle e_i \otimes e_i, \ i \in \mathbb{N}.
\end{align*}
Then, it can be easily seen that $\{G_i\}_{i \in \mathbb{N}}$ is a Parseval HS-frame for $\ell^2(\mathbb{N})$. \\
Next, define $F_i:\ell^2(\mathbb{N}) \to  \mathcal{C}_2(\ell^2(\mathbb{N}))$ as
\begin{align*}
F_i(x)=\frac{199}{200}\langle x, e_i \rangle e_i \otimes e_i+\frac{\sqrt{399}}{200}\langle x, e_i \rangle e_{i+1} \otimes e_{i+1}, \ i \in \mathbb{N}.
\end{align*}
Then,
\begin{align*}
\sum_{i \in \mathbb{N}} \|F_i(x)\|_2^2
&=\sum_{i \in \mathbb{N}} \left\|\frac{199}{200}\langle x, e_i \rangle e_i \otimes e_i+\frac{\sqrt{399}}{200}\langle x, e_i \rangle e_{i+1} \otimes e_{i+1}\right\|_2^2\\
&=\sum_{i \in \mathbb{N}} \left[\frac{199}{200}\langle x, e_i \rangle e_i \otimes e_i+\frac{\sqrt{399}}{200}\langle x, e_i \rangle e_{i+1} \otimes e_{i+1}, \frac{199}{200}\langle x, e_i \rangle e_i \otimes e_i+\frac{\sqrt{399}}{200}\langle x, e_i \rangle e_{i+1} \otimes e_{i+1}\right]_{\text{tr}}\\
&=\frac{39601}{40000}\sum_{i \in \mathbb{N}}|\langle x, e_i \rangle |^2 [e_i \otimes e_i, e_i \otimes e_i ]_{\text{tr}}
+\frac{199\sqrt{399}}{40000}\sum_{i \in \mathbb{N}}|\langle x, e_i \rangle |^2 [e_i \otimes e_i, e_{i+1} \otimes e_{i+1} ]_{\text{tr}}\\
&+\frac{199\sqrt{399}}{40000}\sum_{i \in \mathbb{N}}|\langle x, e_i \rangle |^2 [ e_{i+1} \otimes e_{i+1}, e_i \otimes e_i ]_{\text{tr}}+\frac{399}{40000}\sum_{i \in \mathbb{N}}|\langle x, e_i \rangle |^2 [e_{i+1} \otimes e_{i+1}, e_{i+1} \otimes e_{i+1} ]_{\text{tr}}\\
&=\frac{39601}{40000}\sum_{i \in \mathbb{N}}|\langle x, e_i \rangle |^2 \langle e_i, e_i \rangle \langle e_i, e_i \rangle
+\frac{199\sqrt{399}}{40000}\sum_{i \in \mathbb{N}}|\langle x, e_i \rangle |^2 \langle e_i, e_{i+1} \rangle \langle e_{i+1}, e_i \rangle\\
&+\frac{199\sqrt{399}}{40000}\sum_{i \in \mathbb{N}}|\langle x, e_i \rangle |^2 \langle e_{i+1}, e_i \rangle \langle e_i, e_{i+1} \rangle+\frac{399}{40000}\sum_{i \in \mathbb{N}}|\langle x, e_i \rangle |^2\langle e_{i+1}, e_{i+1} \rangle \langle e_{i+1}, e_{i+1} \rangle\\
&=\frac{39601}{40000}\sum_{i \in \mathbb{N}}|\langle x, e_i \rangle |^2+0+\frac{399}{40000}\sum_{i \in \mathbb{N}}|\langle x, e_i \rangle |^2+0\\
&=\frac{39601}{40000}\|x\|^2 +\frac{399}{40000}\|x\|^2 \\
&=\|x\|^2 \ \text{for all} \ x \in \ell^2(\mathbb{N}).
\end{align*}
This shows that $\{F_i\}_{i \in \mathbb{N}}$ is a Parseval HS-frame for $\ell^2(\mathbb{N})$. \\
For any $x \in \mathcal{H}$, we compute
\begin{align*}
&\text{Re} \sum_{i \in \mathbb{N}} [ F_i(x), G_i(x) ]_{\text{tr}}\\
&=\text{Re} \sum_{i \in \mathbb{N}}\left(\frac{199}{200} \big[\langle x, e_i \rangle e_i \otimes e_i, \langle x, e_i \rangle e_i \otimes e_i \big]_{\text{tr}}+\frac{\sqrt{399}}{200}\big[\langle x, e_i \rangle e_{i+1} \otimes e_{i+1} , \langle x, e_i \rangle e_i \otimes e_i \big]_{\text{tr}}\right)\\
&=\text{Re} \sum_{i \in \mathbb{N}}\left(\frac{199}{200} |\langle x, e_i \rangle |^2[e_i \otimes e_i, e_i \otimes e_i ]_{\text{tr}}+\frac{\sqrt{399}}{200}|\langle x, e_i \rangle |^2[e_{i+1} \otimes e_{i+1} , e_i \otimes e_i ]_{\text{tr}}\right)\\
&=\text{Re} \sum_{i \in \mathbb{N}}\left(\frac{199}{200} |\langle x, e_i \rangle |^2 \langle e_i, e_i \rangle \langle e_i, e_i \rangle+\frac{\sqrt{399}}{200}|\langle x, e_i \rangle |^2  \langle e_{i+1}, e_i \rangle \langle e_i, e_{i+1} \rangle\right)\\
&=\text{Re} \sum_{i \in \mathbb{N}}\left(\frac{199}{200} |\langle x, e_i \rangle |^2 +0\right)\\
&=\frac{199}{200} \|x\|^2.
\end{align*}
This implies that the condition \eqref{eqa1} of Theorem \ref{th5} is satisfied for $\delta=0.005$. Therefore, corresponding to every $\sigma \subset \mathbb{N}$, the weaving $\{F_i\}_{i \in \sigma} \cup \{G_i\}_{i \in \sigma^c}$ is an HS-frame for $\ell^2(\mathbb{N})$ with frame bounds:
\begin{align*}
A = 1 - 2\sqrt{2\delta}=0.8 \quad \text{and} \quad B = 1 + 2\sqrt{2\delta}=1.2.
\end{align*}
\end{exa}

\section{Conclusion}

This paper presents new results on the stability of Hilbert-Schmidt frames under finite perturbations, motivated by real-world scenarios where only a limited number of frame elements are affected. Unlike classical $\mu$-perturbation results, such as those in \cite{CGK}, which assume uniform perturbation across all frame elements and lead to full weaving, our results address a more practical and common setting in applications. In real-world systems such as wireless communication networks, sensor arrays, and distributed data processing, errors and data loss are typically localized due to packet drop, hardware failure, or transient noise affecting only a small number of frame elements. Our results model this localized distortion by permitting finite changes in frame elements and providing explicit bounds under which the perturbed sequence remains an HS-frame. This localized perturbation model not only reflects actual operational challenges more accurately but also ensures that signal reconstruction and processing systems remain stable and reliable even in the presence of partial data loss.
By bridging the gap between theory and practical implementation, our results contribute a robust and realistic framework for the design of resilient operator-valued frame systems, enhancing their applicability in real-time and fault-tolerant signal processing environments. In addition, we extend our analysis to the infinite setting by introducing decay conditions on the perturbations. These results demonstrate that stable weaving can still be achieved even when infinitely many frame elements are perturbed, provided the perturbations decay sufficiently fast (e.g., polynomially or exponentially).

\mbox{}

\begin{thebibliography}{99}\baselineskip10pt

\bibitem{Balaza}
Balazs, P.: Hilbert--Schmidt operators and frames--classification, best approximation by multipliers and algorithms,\emph{ Int. J. Wavelets Multiresolut. Inf. Process.}, \textbf{6} (2),  315--330, (2008).
\bibitem{CK}
Casazza, P. G.,  Christensen, O.: Perturbation of operators and applications to frame theory,  \emph{J. Fourier Anal. Appl.},  3(5),  543--557 (1997).


\bibitem{DVV17}
Deepshikha,  Vashisht, L. K., and Verma, G.:  Generalized weaving frames for  operators in Hilbert spaces, \emph{ Results  Math.}, \textbf{72}, No. 3,  1369--1391,  (2017).

\bibitem{DV17ar}
Deepshikha, Vashisht, L. K.:  Weaving $K$-frames in Hilbert spaces, \emph{Results Math.},  \textbf{73},  No. 2, Art. 81, 20 pp, (2018).

\bibitem {CGK}
Choubin, M., Ghaemi, M. B.,   Kim,  G. H.:  Hilbert--Schmidt frames: Duality, weaving and stability, \emph{Int. J. Nonlinear Anal.  Appl.},  \textbf{11}(1), Article 4251, (2020).

\bibitem{DS}
Duffin,  R. J.,  Schaeffer, A. C.: A class of nonharmonic Fourier series,  \emph{Trans. Amer. Math. Soc.}, \textbf{72}, 341--366,  (1952).


\bibitem{FZUS}
Favier, S. J.,  Zalik, R. A.: On the stability of frames and Riesz bases, \emph{Appl. Comp. Harmon. Anal.},   2(2),  160--173, (1995).

\bibitem{G46}
Gabor, D.: Theory of communication, \emph{ J. Inst. Elect. Eng.}, \textbf{93}, 429--457, (1946).


\bibitem{LG}
 G\v{a}vruta, L.:  Frames for operators,  \emph{ Appl. Compu. Harmon. Anal.},  \textbf{32}, 139--144, (2012).
\bibitem{GossonI}
	 de Gosson, M.~A., Gr\"{o}chenig, K., Romero, J.~L.: Stability of Gabor frames under small time Hamiltonian evolutions, \emph{Lett. Math. Phys.},  106 (6), 799--809 (2016).
\bibitem{GossonII}
de Gosson, M. A.: Symplectic Methods in Harmonic Analysis and in Mathematical Physics, Birkh\"auser, Basel (2011).
\bibitem{KG}
Gr\"{o}chenig, K.: Foundations of Time-frequency Analysis  Birkh\"{a}user, Boston (2001).
\bibitem{BHan}
 Han, B.: Framelets and Wavelets: Algorithms, Analysis, and Applications,  Birkh\"auser,  Basel (2017).

\bibitem{H11}
Heil, C.: A Basis Theory Primer (Expanded ed.) Birkh\"{a}user,  New York (2011).
\bibitem{DivI}
Jindal, D., Jorgensen, P.,   Vashisht,  L.~K: Perturbation of frames from the Weyl-Heisenberg group and the extended affine group, \emph{ Complex Anal. Oper. Theory},  19 (3),  Art. No.  54,  33 pp, (2025).
\bibitem{J}
Jyoti: A note on Hilbert-Schmidt frames for $\mathcal{C}_2(\mathcal{H})$, \emph{J. Anal.}, \textbf{33}, 2883--2895, (2025).

\bibitem{JVPDF1}
Jyoti, Vashisht, L.~K.:  Hilbert-Schmidt frames and Riesz bases with respect to tensor product of Hilbert spaces, \emph{Anal. Math.}, 51 (3), 843--865. 


\bibitem{JV23}
Jyoti, Vashisht, L.~K.:  On Hilbert-Schmidt frames for operators and Riesz bases, \emph{J. Math. Phy. Anal. Geom.},  \textbf{19}, 779--821, (2023).

\bibitem{Pert1}
Jyoti, Malhotra, H.K.: On perturbation of matrix-valued Riesz bases over LCA groups, \emph{Int. J. Wavelets Multiresolut. Inf. Process.},   \textbf{23} (2), Paper No. 2450057,   1--12,  (2025).
\bibitem{Kato}
Kato, T.:  Perturbation Theory for Linear Operators.  Springer-Verlag, Berlin Heidelberg (1976).
\bibitem{Yoo}
Koo, Y. Y.,  Lim, J. K.: Schatten-class operators and frames,  \emph{Quaest. Math.},  \textbf{34}, 203--211, (2011).

\bibitem {SA}
Sadeghi, G.,  Arefijamaal, A.:  Von Neumann-Schatten frames in separable Banach spaces, \emph{Mediterr. J. Math.},  \textbf{42}, 525–535, (2012).

\bibitem{SCHI}
Schatten, R.:  Norm Ideals of Completely Continious Operators,  Springer, Berlin-Heidelberg (1960).

\bibitem{B.simon}
Simon, B.:  Trace Ideals and their Applications,  2$^{nd}$ ed., Mathematical Surveys and Monographs, 120. Amer.  Math.  Soc., Providence, RI (2005).

\bibitem{WsunI}
Sun, W.: $G$-frames and {$g$}-{R}iesz bases,  \emph{J. Math. Anal. Appl.},  \textbf{322} (1), 437--452 (2006).

\bibitem{WsunII}
Sun, W.: Stability of {$g$}-frames,  \emph{ J. Math. Anal. Appl.},  \textbf{326}  (2),  858--868, (2007).

\bibitem{VD3}
 Vashisht, L. K., Deepshikha: Weaving properties of generalized continuous frames generated  by an iterated
function system, \emph{ J. Geom. Phys.},  \textbf{110}, 282--295, (2016).
\bibitem{XZG}
Xiao, X.,  Zhu, Y., G\v{a}vruta, L.: Some properties of $K$-frames in Hilbert spaces, \emph{ Results  Math.},  \textbf{63} (3-4),  1243--1255, (2013).
 \bibitem{Young}
 Young, R. M.: An Introduction to Nonharmonic Fourier Series. Academic Press, New York (1980).

\bibitem{WZhang}
Zhang, W.: Dual and approximately dual Hilbert--Schmidt frames in Hilbert spaces, \emph{ Results  Math.},  \textbf{73}, Art. No. 4, 20pp, (2018).
\bibitem{Zpelle}
Zhang, Z.,   Jorgensen, P. E. T.: Frame Theory in Data Science,  Springer, Cham, 2024.
\end{thebibliography}
\end{document}